\documentclass[a4paper,reqno]{amsart}
\usepackage[margin=1in]{geometry}
\usepackage{amsmath}
\usepackage{float}
\usepackage{cancel,sansmath,bm}
\usepackage{amssymb,mathtools,booktabs,bbm}
\usepackage{enumitem,multicol,cancel,setspace}
\allowdisplaybreaks

\renewcommand{\d}{\boldsymbol{\mathbf{d}}}

\newcommand{\N}{\mathbbm{N}}							
\newcommand{\Z}{\mathbbm{Z}}							
\newcommand{\Q}{\mathbbm{Q}}							
\newcommand{\R}{\mathbbm{R}}							
\newcommand{\D}{\mathcal{D}}
\newcommand{\A}{\mathcal{A}}

\newcommand{\brc}[1]{\left\{ #1 \right\}}
\newcommand{\abs}[1]{\left| #1 \right|}
\newcommand{\pren}[1]{\left( #1 \right)}

\newcommand{\ol}[1]{\overline{#1}}

\newcommand{\mbf}[1]{\mathbf{#1}}

\newcommand*\tageq{\refstepcounter{equation}\tag{\theequation}}

\let\oldlim\lim

\renewcommand{\lim}[3]{\oldlim\limits_{#1 \to #2}{#3}}

\setlist{leftmargin=*}											
\DeclareMathOperator{\tr}{tr}

\DeclareMathOperator{\diag}{diag}

\newtheorem{theorem}{Theorem}
\newtheorem{lemma}[theorem]{Lemma}

\newtheorem{proposition}[theorem]{Proposition}

\theoremstyle{definition}
\newtheorem{definition}{Definition}

\newtheorem*{remark}{Remark}

\newcommand{\e}{\mathbf{e}}

\begin{document}

\author[A.G.R.~Cruz]{Anjelo Gabriel R.~Cruz}
\address[A.G.R.~Cruz]{Institute of Mathematics, University of the Philippines Diliman, 1101 Quezon City, Philippines}
\email{agcruz@math.upd.edu.ph}

\author[M.J.C.~Loquias]{Manuel Joseph C.~Loquias}
\address[M.J.C.~Loquias]{Institute of Mathematics, University of the Philippines Diliman, 1101 Quezon City, Philippines}
\email{mjcloquias@math.upd.edu.ph}

\thanks{A.G.R.~Cruz is grateful to the University of the Philippines System for financial support through its Faculty, REPS, and Administrative Staff Development Program.}



\title{A Matrix Analogue of Rational Number Systems}

\begin{abstract}
	Let $P,Q \in \Z^{d\times d}$ be invertible coprime matrices such that all the eigenvalues of $Q^{-1}P$ have modulus greater than 1,
	and $\Z^d[Q^{-1}P]$ be the smallest non-trivial $Q^{-1}P$-invariant $\Z$-module containing $\Z^d$.
	Suppose there is a finite digit set $\D \subseteq \Z^d + P\Z^d[Q^{-1}P]$ for which every vector $x \in \Z^d + P\Z^d[Q^{-1}P]$ can be represented in the form
	\[
		x = \sum_{i=0}^{\ell-1} (Q^{-1}P)^i Q^{-1}d_i,
	\]
	where the digits $d_i \in \D$ for all $i \in \brc{0,1,\ldots,\ell-1}$. 
	We call such a representation a \textit{$P/Q$-expansion} of $x$, and we say that the digit system $(P,Q,\D)$ has the finiteness property. 
	If, in addition, $\D$ is a complete set of residues of the quotient group $(\Z^d + P\Z^d[Q^{-1}P])/P\Z^d[Q^{-1}P]$, then the digits $d_0, d_1, \dots, d_{\ell-1}$ in the $P/Q$-expansion of $x$ are unique whenever $\ell \in \Z^+$ is minimal, and the resulting digit system is said to have the uniqueness property. We present sufficient conditions for the existence of a digit set $\D$ in which $(P,Q,\D)$ has the finiteness property. For $d=2$, we make use of finite automata to construct digit systems $(P,Q,\D)$ having both the finiteness and uniqueness properties. 
\end{abstract}

\maketitle

\section{Introduction}

\subsection{Context of the paper} In 2008, Akiyama et al.~\cite{frac-akiyama} introduced a representation of positive integers in the rational number base $p/q$, where $p$ and $q$ are coprime positive integers with $p > q$. In this system, every positive integer $N$ can be written uniquely in the form \[N = \sum_{i=0}^k \dfrac{a_i}{q} \pren{\dfrac{p}{q}}^i\] with digits $a_0,a_1,\dots,a_k \in \brc{0,1,\dots,p-1}$ with $a_k \neq 0$. This representation extends to the real numbers by allowing negative integer powers of the base $p/q$ and infinite digit expansions. Such expansions differ from the so-called $\beta$-expansions first studied by R\'{e}nyi \cite{renyi} in 1957, which seek to represent real numbers with respect to a real base $\beta > 1$. 

Akiyama et al.~\cite{frac-akiyama} established several properties of this digit system that are relevant to well-known problems in combinatorics and number theory. 

In particular, the constant $K(p)$ considered by Odlyzko and Wilf \cite{odlyzko} in their study of the \textit{Josephus problem} (see \cite{halbeisen} and  \cite{odlyzko}) coincides with the constant $\omega_{\frac{p}{q}}$ when $q=p-1$.
The constant $\omega_{\frac{p}{q}}$ is the value of an infinite $p/q$-expansion to the right of the radix point that is lexicographically maximal 
among all admissible infinite $p/q$-expansions of real numbers. 
Furthermore, expansions in base $p/q$ are connected to \textit{Mahler's problem} \cite{mahler}, which arose from the study of Pisot and Vijayaraghavan's problem on the distribution of rational powers modulo 1 in the interval $[0,1]$.

The study of rational base digit systems has seen increased interest in recent years. For instance, Loquias et al.~\cite{loquias} formulated an analogue of rational based digit systems for polynomials over finite fields and investigated its properties, as well as its connection to a finite field version of Mahler's problem. More recently, Rossi and Thuswaldner \cite{rossi-thus} employed similar methods to express numbers in base $-3/2$ and studied the associated tiling properties. Knuth \cite{knuth} subsequently examined the representation space $\mathbbm{K}=\mathbbm{R}\times\mathbbm{Q}_2$ introduced in \cite{rossi-thus} and referred to its elements as ambidextrous numbers or ambinumbers.

In addition to digit systems with a rational number base, numeration systems whose base is an invertible square matrix $A$ with rational entries have also been studied.
In particular, Jankauskas and Thuswaldner determined necessary and sufficient conditions for the existence of such systems in \cite{jankauskas, jankauskas-ii}, with the latter providing a more detailed construction of these systems. Moreover, Rossi et al.~\cite{rossi} studied so-called \textit{rational self-affine tiles} associated with matrix digit systems.

The aim of this paper is to present an analogue of the $p/q$ number system for matrices of the form $Q^{-1}P$, where $P,Q \in \Z^{d\times d}$ are suitably chosen invertible matrices. In particular, we determine conditions on $P$ and $Q$ that ensure the existence of a finite subset $\D$ of $\Q^d$ such that a given vector $x$ can be expressed in the form
\[
	x = \sum_{i=0}^{\ell-1} (Q^{-1}P)^i Q^{-1}d_i
\]
with \textit{base} $Q^{-1}P$ and \textit{digits} $d_i \in \D$ for all $i \in \brc{0,1,\dots,\ell-1}$. We show that a suitable choice for $\D$ is any finite set containing a complete residue system of the quotient group $(\Z^d + P\Z^d[Q^{-1}P])/P\Z^d[Q^{-1}P]$, where $\Z^d[Q^{-1}P]$ is the smallest non-trivial $Q^{-1}P$-invariant $\Z$-module containing $\Z^d$. 

These results are presented in Section 2, following the approach of Jankauskas and Thuswaldner in~\cite{jankauskas-ii}.
In Section 3, we then construct finite letter-to-letter transducers that realize addition and subtraction by the canonical basis vectors of $\Z^2$ for certain choices of $P,Q \in \Z^{2\times 2}$ in the vein of Thuswaldner's results in \cite{thuswaldner}. 

\subsection{Number systems with rational base} We begin by recalling the $p/q$ number system studied by Akiyama et al.~ in \cite{frac-akiyama} and the algorithm they introduced to obtain digit expansions with respect to this base. If $p$ and $q$ are coprime positive integers with $p > q$, then every positive integer $N$ can be written uniquely in the form
\begin{equation} \label{eq:pq_exp}
    N = \sum_{i=0}^{k} \dfrac{a_i}{q} \pren{\dfrac{p}{q}}^i
\end{equation}
where $k \in \N$ and $a_0, a_1, \dots, a_{k} \in \brc{0,1,\dots,p-1} =: \D_p$ with $a_{k} \neq 0$. The unique finite word $a_k a_{k-1} \cdots a_0$ over $\D_p$ obtained from (\ref{eq:pq_exp}) is called the \textit{$p/q$-expansion} of $N$ with \textit{digits} $a_0,a_1,\dots,a_k$. One notable difference between the expansion in (\ref{eq:pq_exp}) and both the classical expansion in the integer base~$p$ and generalized $\beta$-expansions is that the digits $a_i$ are divided by $q$.

The digits $a_0, a_1, \dots, a_{k}$ in (\ref{eq:pq_exp}) are obtained through a modified division algorithm.
We start by dividing $qN$ by $p$, obtaining a nonnegative quotient $N_1 < N$ and a remainder $a_0 \in \brc{0,1,\dots,p-1}$ such that
$qN = pN_1 + a_0$.
This process is then applied to $qN_1$, yielding a nonnegative quotient $N_2 < N_1$ and a remainder $a_1 \in \brc{0,1,\dots,p-1}$ satisfying $qN_1 = pN_2 + a_1$. Continuing in this manner produces a decreasing sequence of nonnegative integers $(N_j)_{j\geq 0}$ (where $N_0 := N$) and a sequence of remainders $(a_j)_{j\geq 0}$ in $\D_p$ such that 
\begin{equation} \label{eq:md_algo}
qN_j = pN_{j+1} + a_j
\end{equation}
for all $j\geq 0$. Because $(N_j)_{j\geq 0}$ is a decreasing sequence of nonnegative integers, there exists a minimal index $k$ such that $N_k \neq 0$ and $N_j = 0$ for each $j > k$. 
Solving for each $N_j$ in (\ref{eq:md_algo}) for $0\leq j\leq k$ and iteratively substituting back into the expression for $N_0$ gives the expansion of $N$ in (\ref{eq:pq_exp}).

The modified division algorithm described above can be viewed as the repeated iteration of the dynamical system $T : \Z \to \Z$ defined by 
\begin{equation} \label{eq:quot_func}
    T(n) = \dfrac{qn-d(qn)}{p}
\end{equation}
where $d(k) \in \D_p$ is the unique integer such that $k \equiv d(k) \pmod{p}$ for $k \in \Z$. In this formulation, $N_j = T^j(N)$ for all $j\geq 0$, and the sequence $\big(T^j(N)\big)_{j\geq 0}$ is eventually zero. The digits of $N$ are given by $a_j = d(qN_j)$ for $0\leq j\leq k$.

\subsection{Rational matrix digit systems}\label{RMDS}
 In this subsection, we recall matrix digit systems and state some results established by Jankauskas and Thuswaldner in \cite{jankauskas,jankauskas-ii}. Let $A$ be an invertible $d\times d$ matrix with rational entries. The set $\Z^d[A]$ is defined by
\[
\Z^d[A] = \bigcup_{k=1}^\infty (\Z^d + A\Z^d + \cdots + A^{k-1}\Z^d).
\]
Equivalently, $\Z^d[A]$ consists of all vectors $x \in \Q^d$ of the form
\[
    x = x_0 + Ax_1 + \cdots + A^{k-1}x_{k-1}    
\]
for some $k \in \N$ and $x_0, x_1, \dots, x_{k-1} \in \Z^d$. It is easy to see that $\Z^d[A]$ is the smallest nontrivial $A$-invariant $\Z$-module containing $\Z^d$. 

If $\D$ is a finite subset of $\Z^d[A]$, then the pair $(A,\D)$ is called a (\emph{matrix}) \textit{digit system}. The matrix $A$ is called the \textit{base} and $\D$ is called the \textit{digit set}.

The main objective of \cite{jankauskas-ii} is to determine whether there exists a finite digit set $\D \subseteq \Z^d[A]$ for which $\Z^d[A] = \D[A]$, where 
\[
\D[A] = \bigcup_{k=1}^\infty (\D + A\D + \cdots + A^{k-1}\D).
\]
If such a digit set $\D$ exists, then every vector $x \in \Z^d[A]$ can be written in the form
\begin{equation} \label{eq:mat_exp}
    x = d_0 + Ad_1 + \cdots + A^{k-1} d_{k-1}
\end{equation}
for some $k \in \N$ and $d_0,d_1,\dots,d_{k-1} \in \D$,
and we say that the digit system $(A,\D)$ has the \textit{finiteness property}.
It is clear that any digit set $\D$ for which $(A,\D)$ has the finiteness property must contain a complete system of residues of the quotient group $\Z^d[A]/A\Z^d[A]$, but the converse does not necessarily hold. If, in addition, $\abs{\D} = \abs{\Z^d[A]/A\Z^d[A]}$, then the expansion of $x \in \Z^d[A]$ in the form (\ref{eq:mat_exp}) is unique whenever $k$ is minimal. In this case, $(A,\D)$ is said to possess the \textit{uniqueness property}. A digit system $(A,\D)$ having both the finiteness and uniqueness properties is called \textit{standard}. For standard digit systems, it is also assumed that $0 \in \D$ in order to perform positional arithmetics. 

The division algorithm in the digit system $(A,\D)$ can be captured through a dynamical system analogous to (\ref{eq:quot_func}). If $\D \subseteq \Z^d[A]$ is a complete residue system modulo $A\Z^d[A]$, then a function $\d : \Z^d[A] \to \D$ is called a \textit{remainder function} or a \textit{digit function} if \[x \equiv \d(x) \pmod{A\Z^d[A]},\] i.e., if $x - \d(x) \in A\Z^d[A]$ for all $x \in \Z^d[A]$. Using $\d$, one defines a dynamical system $\Phi : \Z^d[A] \to \Z^d[A]$, called the \textit{quotient function}, by
\[
\Phi(x) = A^{-1}\big(x - \d(x)\big)
\]
for each $x \in \Z^d[A]$. It follows from the definition of $\Phi$ and an induction argument that for each $n \in \N$, there exist unique $d_0, d_1, \dots, d_{n-1} \in \D$ such that
\begin{equation} \label{eq:attr_exp}
    x = d_0 + Ad_1 + \cdots + A^{n-1}d_{n-1} + A^n \Phi^n(x).
\end{equation}
In fact, $d_j = (\d \circ \Phi^j)(x)$ for all $0\leq j\leq n-1$. If $\Phi^n(x) = 0$ for some $n \in \N$, then $x$ admits a representation of the form (\ref{eq:mat_exp}). Hence, $(A,\D)$ has the finiteness property if and only if for each $x \in \Z^d[A]$, there exists $n \in \N$ such that $\Phi^n(x) = 0$. This characterization leads to the following definition.

\begin{definition} \label{def:attractor}
A subset $\A$ of $\Z^d[A]$ is called an \textit{attractor} of $\Phi$ if the following conditions are satisfied:
\begin{enumerate}[label=\arabic*.]
    \item $\Phi(\A) \subseteq \A$;
    \item for every $x \in \Z^d[A]$, there exists $n = n(x) \in \N$ such that $\Phi^n(x) \in \A$;
    \item no proper subset of $\A$ satisfies (1) and (2).
\end{enumerate}
\end{definition}
It follows immediately from the definition that the attractor of $\Phi$, if it exists, is unique and we denote it by $\A_\Phi$. In general, $\A_\Phi$ may be finite or infinite. However, as shown in \cite{jankauskas-ii}, one can construct a digit system with the finiteness property whenever $\A_\Phi$ is finite \cite[Proposition 2.2]{jankauskas-ii}. In addition, certain conditions on the matrix $A$ guarantee that the attractor of $\Phi$ is finite, and hence $A$ is the base of a digit system with the finiteness property \cite[Proposition 2.5]{jankauskas-ii}. Both results are summarized in the following proposition.

\begin{proposition} \label{prop:fin_prop}
Let $(A,\D)$ be a matrix digit system with quotient function $\Phi$.
\begin{enumerate}[label=\arabic*.]
    \item If the attractor $\A_\Phi$ is a finite set, then the digit system $(A,\D \cup \A_\Phi)$ has the finiteness property.
    \item If every eigenvalue of $A$ has modulus greater than 1, then the attractor of $\Phi$ is finite.
\end{enumerate}
\end{proposition}

A matrix $A$ satisfying the second statement of Proposition \ref{prop:fin_prop} is said to be \textit{expanding}. Thus, any expanding matrix $A$ gives rise to a digit system that has the finiteness property (though not necessarily the uniqueness property).

\section{A Matrix Analogue of Rational Number Digit Systems}

We are now in a position to introduce an analogue of the modified division algorithm that produces a digit system analogous to the rational number based system studied by Akiyama et al.~in~\cite{frac-akiyama}. Following the framework in \cite{jankauskas}, we establish sufficient conditions for the existence of a matrix digit system possessing a form of the finiteness property.

\subsection{Lattices and coprime matrix pairs} We begin this section by recalling some concepts from the theory of lattices, which will be relevant to the discussion that follows. Further background on these topics can be found in \cite{baake, cassels, neukirch}. 

Let $d$ be a positive integer. A \textit{lattice $\Gamma$ in $\R^d$} (of rank and dimension $d$) is the $\Z$-span of $d$ linearly independent vectors in $\R^d$. That is, $\Gamma = \Z v_1 \oplus \Z v_2 \oplus \cdots \oplus \Z v_d$,
where $v_1, v_2, \dots, v_d$ are linearly independent vectors of $\R^d$, and $\Z v := \brc{nv \mid n \in \Z}$. The vectors $v_1, \dots, v_d$ are said to form a \textit{basis} for $\Gamma$.

A lattice may also be defined equivalently as a \textit{discrete, co-compact subgroup} of $\R^d$. That is, $\Gamma$ is a subgroup of $\R^d$ with no accumulation points such that the quotient group $\R^d/\Gamma$ is a compact topological group. Some basic properties of lattices are summarized in the following remark.
\begin{remark} Let $\Gamma$ be a lattice in $\R^d$.
\begin{enumerate}[label=\arabic*.]
    \item Then $\Gamma = A\Z^d$, where $A$ is an invertible $d \times d$ matrix with real entries. In fact, the columns of $A$ form a basis for $\Gamma$; for this reason, $A$ is called a \textit{basis matrix} of $\Gamma$. If $A$ is the $d\times d$ identity matrix, then $\Gamma$ is the \textit{integer lattice} $\Z^d$.
    
    \item If $\Gamma = A\Z^d = B\Z^d$ where $A,B \in \R^{d\times d}$ are invertible, then $A = BU$ for some integer matrix $U$ with determinant $\pm 1$, and conversely. Such a matrix $U$ is said to be \textit{unimodular}. Hence, any two basis matrices of $\Gamma$ have the same determinant in absolute value. In particular, any lattice having a unimodular basis matrix must be the integer lattice $\Z^d$.
    
    \item A \textit{sublattice} $\Gamma'$ of $\Gamma$ is a subgroup of finite index in $\Gamma$. Equivalently, $\Gamma'$ is a subgroup of $\Gamma$ that is itself a lattice in $\R^d$. In particular, the group-theoretic sum of two sublattices of $\Gamma$ is again a sublattice of~$\Gamma$.
    
    \item The lattice $\Gamma$ is \textit{locally finite}, i.e., it intersects any compact subset of $\R^d$ in only finite many points.
\end{enumerate}    
\end{remark}

Now let $P,Q \in \Z^{d \times d}$ be invertible. By Remark 3, the sum $P\Z^d + Q\Z^d$ is a sublattice of $\Z^d$, so there exists an invertible matrix $A \in \Z^{d \times d}$ such that $P\Z^d + Q\Z^d = A\Z^d$. It follows that the columns of $P$ and $Q$ lie in $A\Z^d$. In particular, there exist $P',Q' \in \Z^{d \times d}$ such that $P = AP'$ and $Q = AQ'$, i.e., $P' = A^{-1} P$ and $Q' = A^{-1} Q$ are invertible matrices with integer entries. Furthermore, \[P'\Z^d + Q'\Z^d = \Z^d.\] This observation motivates the following definition, which serves as a matrix analogue of coprime integers.
\begin{definition} \label{def:coprime}
Two invertible integer matrices $P$ and $Q$ are said to be \textit{coprime} if $P\Z^d + Q\Z^d = \Z^d$. 
\end{definition}
A more detailed discussion of coprime matrices, including their properties, constructions, and application to related fields can be found in \cite{pal-i,pal-ii,pal-iii}.

The preceding discussion implies that for any pair of invertible matrices $P, Q \in \Z^{d \times d}$, there is an invertible matrix $A \in \Z^{d \times d}$ such that $A^{-1}P$ and $A^{-1}Q$ are coprime.  Such a matrix is unique up to right multiplication by a unimodular matrix. The following proposition gives some properties of coprime matrices.

\begin{proposition} \label{prop:coprime}
Let $P,Q \in \Z^{d \times d}$.
\begin{enumerate}[label=\arabic*.]
    \item Then $P$ and $Q$ are coprime if and only if there exist $X,Y \in \Z^{d \times d}$ such that $PX + QY = I$.
    \item The matrices $P$ and $Q$ are coprime whenever $\det P$ and $\det Q$ are coprime.
\end{enumerate}
\end{proposition}

\begin{proof}
For the first statement, suppose that $P$ and $Q$ are coprime. Let $\e_i$ denote the $i$th canonical basis vector of $\Z^d$. Then for each $i$, the equation $Px + Qy = \e_i$ has a solution $x_i,y_i \in \Z^d$. Setting $X = (x_1, x_2, \dots, x_d)$ and $Y = (y_1, y_2, \dots, y_d)$ yields $PX+QY=I$. The converse follows immediately from the definition.

For the second statement, let $A \in \Z^{d\times d}$ be a basis matrix of the lattice $P\Z^d + Q\Z^d$. Then $M := A^{-1}P$ and $N := A^{-1}Q$ have integer entries, so that $\det A$ divides both $\det P$ and $\det Q$. If $\det P$ and $\det Q$ are coprime, it follows that $\det A = \pm 1$, i.e., $A$ is unimodular. Therefore $P\Z^d + Q\Z^d = \Z^d$, so $P$ and $Q$ are coprime.
\end{proof}

The converse of the second statement does not hold in general. For example, the diagonal matrices $P = \diag(2,3)$ and $Q = \diag(3,2)$ have the same determinant, yet they are coprime.

\subsection{Remainder and quotient functions in the $P/Q$ system} \label{subsec:rem_div} 
We are now ready to introduce analogues of the remainder and quotient functions discussed in Section~\ref{RMDS} for a base given by the quotient of two matrices.  Let $d$ be a positive integer, and $P,Q \in \Z^{d\times d}$ be invertible coprime matrices. Consider the $\Z$-modules $\Z^d[Q^{-1}P]$ and $\Z^d[PQ^{-1}]$. It follows readily that $P\Z^d[Q^{-1}P]$ is a submodule of $\Z^d[PQ^{-1}]$. We claim that the coprimality of $P$ and $Q$ implies that	$P\Z^d[Q^{-1}P] = (PQ^{-1})\Z^d[PQ^{-1}]$, and consequently, $\Z^d + P\Z^d[Q^{-1}P] = \Z^d[PQ^{-1}]$. It then follows from \cite[Lemma 2.3]{jankauskas-ii}, with $A=PQ^{-1}$, that the abelian quotient group $(\Z^d + P\Z^d[Q^{-1}P]) / P\Z^d[Q^{-1}P]$ is finite.

Indeed, let 
\[v = P \sum\limits_{i=0}^{\ell-1} (Q^{-1}P)^ix_i \in P\Z^d[Q^{-1}P].\] Then
\begin{align*}
    v &= \sum\limits_{i=0}^{\ell-1} (PQ^{-1})^i(Px_i) = \sum\limits_{i=0}^{\ell-1} (PQ^{-1})^{i+1}(Qx_i) \in (PQ^{-1}) \Z^d[PQ^{-1}].
\end{align*}
Conversely, suppose \[w = \sum\limits_{i=0}^{\ell-1} (PQ^{-1})^{i+1}x_i \in (PQ^{-1}) \Z^d[PQ^{-1}].\] Because $P$ and $Q$ are coprime, we may write $x_i = Py_i + Qz_i$ for some $y_i, z_i \in \Z^d$, for each $i \in \brc{0,1,\dots,\ell-1}$. Thus,
\begin{align*}
    w &= \sum\limits_{i=0}^{\ell-1} (PQ^{-1})^{i+1}(Py_i + Qz_i)\\
    &= \sum\limits_{i=0}^{\ell-1} P(Q^{-1}P)^{i+1} y_i + \sum\limits_{i=0}^{\ell-1} (PQ^{-1})^i (Pz_i) \\
    &= P\pren{\sum_{i=0}^{\ell-1} (Q^{-1}P)^{i+1} y_i + \sum_{i=0}^{\ell-1} (Q^{-1}P)^i z_i} \in P\Z^d[Q^{-1}P].
\end{align*}
This proves the claim.

Our goal is to represent every vector in $\Z^d[Q^{-1}P]$ in a form analogous to (\ref{eq:pq_exp}), using a method similar to the modified division algorithm. Throughout, we assume that $\D$ is a finite set containing a complete residue system of $(\Z^d + P\Z^d[Q^{-1}P]) / P\Z^d[Q^{-1}P]$. A mapping $\d : \Z^d + P\Z^d[Q^{-1}P] \to \D$ is called a \textit{digit} or \textit{remainder function} if it satisfies
\[
\d(v) \equiv v \pmod{P\Z^d[Q^{-1}P]}
\]
i.e., $v - \d(v) \in P\Z^d[Q^{-1}P]$, for each $v \in \Z^d + P\Z^d[Q^{-1}P]$. 

We define the dynamical system $\Phi : \Z^d[Q^{-1}P] \to \Z^d[Q^{-1}P]$ associated to $\d$ by \[\Phi(x) = P^{-1}\big(Qx - \d(Qx)\big).\] Observe that $\Phi$ is well defined since $Qx \in \Z^d + P\Z^d[Q^{-1}P]$ and $Qx - \d(Qx) \in P\Z^d[Q^{-1}P]$ for all $x \in \Z^d[Q^{-1}P]$. Moreover, $\Phi$ may be viewed as a matrix analogue of the dynamical system $T$ defined in (\ref{eq:quot_func}). 
Arguing as in \cite{jankauskas-ii}, we obtain that for every $x \in \Z^d[Q^{-1}P]$ and $n\in\N$, there exist unique $d_0, d_1, \dots,d_{n-1} \in \D$ such that
\begin{equation} \label{eq:pq_dyn_exp}
    x = (Q^{-1}P)^n \Phi^n(x) + \sum_{i=0}^{n-1} (Q^{-1}P)^i Q^{-1}d_i.
\end{equation}
Here, $d_i := \d\big(Q\Phi^i(x)\big)$ for each $0\leq i\leq n-1$. If $\Phi^n(x) = 0$ for some minimal $n \in \N$, then (\ref{eq:pq_dyn_exp}) yields an expression for $x$ that is analogous to the base-$p/q$ expansion in (\ref{eq:pq_exp}). The resulting digit system is denoted by $(P,Q,\D)$. We say that the digit system $(P,Q,\D)$ has the \textit{finiteness property} if for all $x \in \Z^d[Q^{-1}P]$, there exist $d_0,d_1,\dots,d_{\ell-1} \in \D$ such that
\begin{equation} \label{eq:pq_mat_exp}
x = \sum\limits_{i=0}^{\ell-1} (Q^{-1}P)^i Q^{-1}d_i.
\end{equation}
If $\ell$ is minimal and $d_0, d_1, \dots, d_{\ell-1}$ are unique, then the digit system $(P,Q,\D)$ is said to have the \textit{uniqueness property}. The expression in (\ref{eq:pq_mat_exp}) is called the \textit{$P/Q$-expansion of $x$} with \textit{digits} $d_0,d_1,\dots,d_{\ell-1}$. The digit system $(P,Q,\D)$ is said to be \textit{standard} if it possesses both the finiteness and uniqueness properties. For such digit systems, we assume that $0 \in \D$ to allow padding with leading zeros when performing positional arithmetic.

Although the expansion in (\ref{eq:pq_dyn_exp}) appears to differ from that in (\ref{eq:attr_exp}) due to the multiplication of the digits by the matrix $Q^{-1}$, the following result shows that the digit system $(P,Q,\D)$ coincides with the digit system $(Q^{-1}P, Q^{-1}\D)$ studied in \cite{jankauskas-ii}.

\begin{proposition} \label{prop:digit_system}
Let $P,Q \in \Z^{d\times d}$ be coprime and $\D$ be a complete residue system of the quotient group $(\Z^d + P\Z^d[Q^{-1}P]) / P\Z^d[Q^{-1}P]$. Then the following statements hold.
\begin{enumerate}[label=\arabic*.]
	\item The set $Q^{-1}\D$ is a complete residue system of $\Z^d[Q^{-1}P] / (Q^{-1}P) \Z^d[Q^{-1}P]$.
	\item Let $\d' : \Z^d[Q^{-1}P] \to Q^{-1}\D$ and $\Phi' : \Z^d[Q^{-1}P] \to \Z^d[Q^{-1}P]$ denote the digit function and the associated dynamical system of the digit system $(Q^{-1}P,Q^{-1}\D)$, respectively. Then $\Phi' = \Phi$, where $\Phi$ is the dynamical system associated with the digit system $(P,Q,\D)$.
\end{enumerate}
\end{proposition}

\begin{proof}
We begin by proving the first statement. Because $\D$ is a complete residue system of the quotient group $(\Z^d + P\Z^d[Q^{-1}P]) / P\Z^d[Q^{-1}P]$, we have \[\D + P\Z^d[Q^{-1}P] = \Z^d + P\Z^d[Q^{-1}P].\] Multiplying by $Q^{-1}$ and using the coprimality of $P$ and $Q$, we obtain

\begin{align*}
	Q^{-1}\D + (Q^{-1}P)\Z^d[Q^{-1}P] &= Q^{-1}\Z^d + (Q^{-1}P)\Z^d[Q^{-1}P]  \\
	&= \Z^d + (Q^{-1}P)\Z^d + (Q^{-1}P)\Z^d[Q^{-1}P] \\
	&= \Z^d + (Q^{-1}P)\Z^d[Q^{-1}P] \\
	&= \Z^d[Q^{-1}P].
\end{align*}
This proves the first statement. 

For the second statement, recall that $\d' : \Z^d[Q^{-1}P] \to Q^{-1}\D$ satisfies \[x \equiv \d'(x) \pmod{(Q^{-1}P)\Z^d[Q^{-1}P]}\] for all $x \in \Z^d[Q^{-1}P]$ and $\Phi' : \Z^d[Q^{-1}P] \to \Z^d[Q^{-1}P]$ is given by \[\Phi'(x) = P^{-1}Q\big(x - \d'(x)\big).\] Let $\d : \Z^d + P\Z^d[Q^{-1}P] \to \D$ be the digit function of the system $(P,Q,\D)$. Given $x\in \Z^d[Q^{-1}P]$, we have
\begin{align*}
	Qx \equiv \d(Qx) \pmod{P\Z^d[Q^{-1}P]} &\iff x \equiv Q^{-1}\d(Qx) \pmod{(Q^{-1}P)\Z^d[Q^{-1}P]}.
\end{align*}
Since $Q^{-1}\d(Qx) \in Q^{-1}\D$ and $x \equiv \d'(x) \pmod{(Q^{-1}P)\Z^d[Q^{-1}P]}$, it follows that $\d(Qx) = Q\d'(x)$. Therefore,
\[
	\Phi(x) = P^{-1}\big(Qx - \d(Qx)\big) = P^{-1}Q\big(x - \d'(x)\big) = \Phi'(x),
\]
as desired. This completes the proof.
\end{proof}

By Proposition \ref{prop:digit_system}, Proposition \ref{prop:fin_prop} carries over to the digit system $(P,Q,\D)$. We state this explicitly in the next Proposition.
Note that the attractor $\A_{\Phi}$ of the dynamical system $\Phi$ associated with the digit system $(P,Q,\D)$ is contained in $\Z^d[Q^{-1}P]$.

\begin{proposition} Let $P,Q \in \Z^{d\times d}$ be coprime, and $\Phi$ be the dynamical system associated with the digit system $(P,Q,\D)$. \label{prop:fin_attractor}
\begin{enumerate}[label=\arabic*.]

	\item If the attractor $\A_\Phi$ is a finite set, then the digit system $(P,Q,\D \cup Q\A_{\Phi})$ has the finiteness property.
	\item If $Q^{-1}P$ is expanding, then the attractor of $\Phi$ is finite.
\end{enumerate}
\end{proposition}

\subsection[Expansions of integer vectors]{Expansions of integer vectors} We now focus our attention on finding expansions of vectors in $\Z^d$. For the matrix base $A$, the process of finding expansions of vectors in $\Z^d$ is done in \cite{jankauskas-ii} by considering a complete residue system $\D$ of $\Z^d/(\Z^d \cap A\Z^d)$. A \textit{restricted digit function} $\d_r : \Z^d \to \D$ is then defined by mapping $x \in \Z^d$ to its residue $\d_r(x) \in \D$ modulo the \textit{auxiliary lattice} $\Z^d \cap A\Z^d$. The corresponding dynamical system $\Phi_r : \Z^d \to \Z^d$ is then defined in a similar manner: \[\Phi_r(x) = A^{-1}\big(x - \d_r(x)\big).\]
The resulting mapping $\Phi_r$ allows one to perform remainder division while also preserving $\Z^d$, which is generally not the case when obtaining remainders modulo $A\Z^d[A]$. Moreover, $\Z^d/(\Z^d \cap A\Z^d)$ contains an isomorphic copy of $\Z^d[A]/A\Z^d[A]$ by \cite[Lemma 2.7]{jankauskas-ii} and so $\D$ can also be chosen to contain a complete set of coset representatives of $\Z^d[A]/A\Z^d[A]$.

One drawback of the remainder division introduced in Subsection \ref{subsec:rem_div} is the fact that the digit set $\D$ is dependent on both $P$ and $Q$. This is not the case for the modified division algorithm in \cite{frac-akiyama}, as the digit set $\D_p$ depends only on the numerator $p$. In light of this situation, our aim is to find expansions of vectors $x \in \Z^d$ having the form (\ref{eq:pq_mat_exp}) whose digits $d_i$ are representatives of the quotient group $\Z^d/P\Z^d$. This approach is similar to the restricted division in \cite{jankauskas-ii} but is ultimately different: an expansion of $x \in \Z^d$ of the form
\[
x = \sum_{i=0}^{n-1} (Q^{-1}P)^i Q^{-1}d_i
\]
where each $d_i \in \Z^d$ belongs to a complete residue system $\D$ of $\Z^d/P\Z^d$ may not necessarily be an integer expansion of the same vector $x$ when division is done modulo the auxiliary lattice $\Z^d/(\Z^d \cap Q^{-1}P\Z^d)$. In fact, $Q^{-1}\D$ may not necessarily consist of integers. That said, $\Z^d/P\Z^d$ and $\Z^d/(\Z^d \cap Q^{-1}P\Z^d)$ are isomorphic as groups whenever $P$ and $Q$ are coprime.

We begin with the following result, which relates the residue group $\Z^d/P\Z^d$ in our current setting with the group $(\Z^d + P\Z^d[Q^{-1}P])/ P\Z^d[Q^{-1}P]$ we considered in the previous section.

\begin{lemma} \label{lem:subgp}
Let $P,Q \in \Z^{d\times d}$ be coprime. Then $(\Z^d + P\Z^d[Q^{-1}P])/ P\Z^d[Q^{-1}P]$ is isomorphic to a subgroup of $\Z^d / P\Z^d$. If the determinants of $P$ and $Q$ are coprime, then
\[
\dfrac{\Z^d + P\Z^d[Q^{-1}P]}{P\Z^d[Q^{-1}P]} \cong \dfrac{\Z^d}{P\Z^d}.
\]
Furthermore, any complete residue system of $\Z^d/P\Z^d$ is also a complete residue system of $(\Z^d + P\Z^d[Q^{-1}P])/P\Z^d[Q^{-1}P]$.
\end{lemma}

\begin{proof}
We prove only the second statement as the first is proven in a manner similar to \cite[Lemma 2.7]{jankauskas-ii}, using the Third Isomorphism Theorem together with the fact that $P\Z^d \leq \Z^d \cap P\Z^d[Q^{-1}P] \leq \Z^d$. Suppose $P$ and $Q$ have coprime determinants and let 
\[
m := \abs{\dfrac{\Z^d[PQ^{-1}]}{(PQ^{-1}) \Z^d[PQ^{-1}]}} = \abs{\dfrac{\Z^d + P\Z^d[Q^{-1}P]}{P\Z^d[Q^{-1}P]}}, \text{ and }
n := \abs{\dfrac{\Z^d[QP^{-1}]}{(QP^{-1}) \Z^d[QP^{-1}]}}.
\]
By \cite[Proposition 2.2]{rossi}, $\abs{\det(PQ^{-1})} = \frac{\abs{\det P}}{\abs{\det Q}} = \frac{m}{n}$. Thus, $n \abs{\det P} = m \abs{\det Q}$. Since $\det P$ and $\det Q$ are coprime, it follows that $\abs{\det P}$ divides $m$. By the first statement, $m$ divides $\abs{\det P}$, and so \[\abs{\dfrac{\Z^d + P\Z^d[Q^{-1}P]}{P\Z^d[Q^{-1}P]}} = m = \abs{\det P} = \abs{\dfrac{\Z^d}{P\Z^d}},\] which shows isomorphism. If $\D$ is a complete set of coset representatives of $\Z^d/P\Z^d$, then $\D + P\Z^d = \Z^d$. Thus,
\[
	\D + P\Z^d[Q^{-1}P] = \D + P\Z^d + P\Z^d[Q^{-1}P] = \Z^d + P\Z^d[Q^{-1}P],
\]
i.e., $\D$ contains a complete residue system of $(\Z^d + P\Z^d[Q^{-1}P])/P\Z^d[Q^{-1}P]$. Because $\abs{\D} = \abs{\Z^d/P\Z^d} = \abs{(\Z^d + P\Z^d[Q^{-1}P])/P\Z^d[Q^{-1}P]}$ it follows that $\D$ is a complete set of coset representatives of $(\Z^d + P\Z^d[Q^{-1}P])/P\Z^d[Q^{-1}P]$.
\end{proof}

As mentioned previously, we now focus on introducing an alternative form of division whose remainders are residues of the group $\Z^d/P\Z^d$. This leads to the following definition.

\begin{definition} \label{def:restrict}
Let $\D \subseteq \Z^d$ be a finite digit set that contains a complete set of coset representatives of $\Z^d / P\Z^d$. A \textit{restricted digit function} is a mapping $\d_r : \Z^d \to \D$ satisfying $\d_r(x) \equiv x \pmod{P\Z^d}$. To this restricted digit function, we define the \textit{associated dynamical system} $\Phi_r : \Z^d \to \Z^d$ given by $\Phi_r(x) = P^{-1} \big(Qx - \d_r(Qx)\big)$.
\end{definition}

The attractor $\A_{\Phi_r} \subseteq \Z^d$ of the dynamical system $\Phi_r$ is defined analogously. As in Subsection \ref{subsec:rem_div}, for each $x \in \Z^d$ and each $n \in \Z^+$, one can find digits $d_0,d_1,\dots,d_{n-1} \in \D$ such that
\[
x = (Q^{-1}P)^n \Phi^n(x) + \sum_{i=0}^{n-1} (Q^{-1}P)^i Q^{-1}d_i.
\]

We can extend this restricted division procedure to all of $\Z^d[Q^{-1}P]$ using an approach analogous to the process done by Jankauskas and Thuswaldner in \cite{jankauskas-ii}. If $x \in \Z^d[Q^{-1}P]$, then there is a smallest $k \in \Z^+$ such that $x \in \Z^d_k[Q^{-1}P]$. Hence, $x$ has a smallest expansion of the form
\begin{equation} \label{eq:shortest}
x = x_0 + Q^{-1}Px_1 + \cdots + (Q^{-1}P)^{k-1} x_{k-1}
\end{equation}
where $x_0, x_1, \dots, x_{k-1} \in \Z^+$, not necessarily unique (e.g. we can add any $v \in \Z^d \cap Q^{-1}P\Z^d$ to $x_0$ and subtract $x_1$ by $P^{-1}Qv$). To make such an expansion unique, we can introduce a (lexicographic) order $\mathcal{O}$ on $k$-tuples of vectors in $\Z^d$ and choose an expansion of the form (\ref{eq:shortest}) such that $(x_0, \dots, x_{k-1})$ is smallest with respect to $\mathcal{O}$. 

Let (\ref{eq:shortest}) be the minimal expansion of $x \in \Z^d[Q^{-1}P]$ with respect to $\mathcal{O}$. We now define the dynamical system $\Psi : \Z^d[Q^{-1}P] \to \Z^d[Q^{-1}P]$ by
\begin{align*}
    \Psi(x) := P^{-1}[Qx - \d_r(Qx_0)] &= P^{-1}[Qx + P\Phi_r(x_0) - Qx_0] \\
    &= \Phi_r(x_0) + P^{-1}Q(x - x_0) \\
    &= [\Phi_r(x_0) + x_1] + (Q^{-1}P)x_2 + \cdots + (Q^{-1}P)^{k-2} x_{k-1}.
\end{align*}

Thus $\Psi$ maps $\Z^d_k[Q^{-1}P]$ into $\Z^d_{k-1}[Q^{-1}P]$, and $\Phi_r$ is $\Psi$ restricted to $\Z^d$. The attractor $\A_\Psi$ of $\Psi$ is defined analogously. The following result is an analogue of \cite[Lemma 2.9]{jankauskas-ii}, which shows that it suffices to look at the restricted dynamical system $\Phi_r$ in order to determine the attactor of the extended system~$\Psi$.

\begin{lemma} \label{lem:attractor}
Let $P,Q \in \Z^{d\times d}$ be coprime and $\D \subseteq \Z^d$ be a finite digit set that contains a complete set of coset representatives of $\Z^d / P\Z^d$. If $\Phi_r$ has finite attractor $\A_{\Phi_r}$ in $\Z^d$, then $\A_\Psi = \A_{\Phi_r}$. In particular, if $\Phi_r$ is ultimately zero in $\Z^d$, then $\Psi$ is also ultimately zero in $\Z^d[Q^{-1}P]$, and the digit system $(P,Q, \D \cup Q\A_{\Phi_r})$ has the finiteness property.
\end{lemma}

\begin{proof}
Similar to \cite[Lemma 2.9]{jankauskas-ii}.
\end{proof}

As a consequence of Lemma \ref{lem:attractor}, it suffices to look at the dynamical system $\Phi_r$ acting on $\Z^d$ to obtain a digit system in $\Z^d[Q^{-1}P]$ having the finiteness property. By replacing $\Z^d[Q^{-1}P]$, $\Phi$, and $P\Z^d[Q^{-1}P]$ in Proposition \ref{prop:fin_attractor}, we see that if $Q^{-1}P$ is expanding, then the dynamical system $\Phi_r$ has finite attractor $\A_{\Phi_r}$ in $\Z^d$. By Lemma \ref{lem:attractor}, this gives rise to a digit system in $\Z^d[Q^{-1}P]$ having the finiteness property. This allows us to find radix expansions in $\Z^d[Q^{-1}P]$ while also preserving $\Z^d$.

\section[Attractors of Two-Dimensional $P/Q$ Digit Systems]{Attractors of Two-Dimensional $P/Q$ Digit Systems}

Suppose $(P,Q,\D)$ is a digit system. 
In the preceding sections, we showed that whenever $P,Q \in \Z^{d\times d}$ are coprime and $Q^{-1}P$ is expanding, then the attractor $\A$ of the dynamical system associated with $(P,Q,\D)$ is finite. 
We may then multiply the vectors in $\A$ by $Q$ and adjoin the resulting elements to $\D$ in order to obtain a digit system having the finiteness property. 
Note, however, that this procedure does not always yield a standard digit system when $\A$ is nonzero, since the uniqueness property is not satisfied.

In this section, we determine the attractor of digit systems $(P,Q,\D)$ where $P,Q \in \Z^{2\times 2}$ are invertible matrices of a particular form with coprime determinants 
(which implies by Proposition \ref{prop:coprime} that $P$ and $Q$ are themselves coprime), and $\D$ is a fixed digit set to be defined later. 
In view of Lemma \ref{lem:attractor}, it suffices to work with the restricted dynamical system $\Phi_r$, since extending to the full module $\Z^d[Q^{-1}P]$ preserves the attractor. 

To this end, we employ a technique from \cite{thuswaldner}, corresponding to the case where
\[
P = 
\begin{bmatrix}
    a & b \\
    \varepsilon & d
\end{bmatrix},
\]
with $a,b,d \in \Z$ and $\varepsilon = \pm 1$, while $Q$ is the $2\times 2$ identity matrix. 
A finite letter-to-letter transducer is constructed that takes as input the digit string corresponding to the finite base-$P$ expansion of a vector $v \in \Z^2$ 
and outputs a (possibly infinite, though always left-periodic) digit string corresponding to the expansion of the vectors $v \pm (1,0)^\top$ and $v \pm (0,1)^\top$. 
For a more detailed discussion of languages and automata theory, we refer the reader to \cite{allouche, hopcroft, sakarovitch}. 

Loosely speaking, the attractor of $(P,\D)$ arises from the infinite cycles of the transducer whose edge labels have input digit $0$. 
The construction of this transducer relies on the characteristic polynomial $\chi_P(x) = x^2 + \alpha x + \beta$ of $P$, 
together with the identity \[P^2 + \alpha P + \beta I = 0\] which governs the propagation of carries whene the digit expansion of a vector $v$ is incremented by~$(1,0)^\top$. 

Similar automata have also been studied for rational-base digit systems. 
In~\cite{frac-akiyama}, a transducer is constructed that takes as input the $p/q$-expansion of a positive integer $N$ and outputs that of $N+1$. 
Knuth \cite{knuth} carried out an analogous construction for the base $(-3/2)$ digit system, 
while Frougny and Sakarovitch \cite{berthe-rigo} discussed necessary and sufficient conditions for addition by 1 in the so-called $U$-systems.

\subsection{Introduction and statement of the main result}

Suppose $P,Q \in \Z^{d\times d}$ are coprime and $\D$ is a complete residue system of $\Z^d/P\Z^d$. 
Let $\d_r : \Z^d \to \D$ and $\Phi_r : \Z^d \to \Z^d$ denote the restricted digit and quotient functions arising from the digit system $(P,Q,\D)$ as in Definition \ref{def:restrict}. 
By (\ref{eq:pq_dyn_exp}), every $v \in \Z^d$ can be written in the form
\begin{equation} \label{eq:rep}
    v = (Q^{-1}P)^{k+1}p + \sum_{j=0}^k (Q^{-1}P)^jQ^{-1}a_j
\end{equation}
for some minimal $k \in \N$, some unique attractor element $p$, and unique digits $a_0, \dots, a_k \in \D$ (in this case, $p = \Phi_r^{k+1}(v)$). 
If the attractor of $\Phi_r$ is finite (which is guaranteed when $Q^{-1}P$ is expanding), then there exists a smallest positive integer $t$ such that $\Phi_r^t(p) = p$. 
Applying the same process to $p$, we obtain unique digits $b_0, b_1, \dots, b_{t-1} \in \D$ satisfying 
\[p = (Q^{-1}P)^t p + \sum\limits_{\ell=0}^{t-1} (Q^{-1}P)^\ell Q^{-1} b_\ell.\]
Iterating this equation $N$ times, for arbitrarily large $N$, yields
\begin{equation} \label{eq:attr}
    p = (Q^{-1}P)^{Nt} p + \sum_{\ell=0}^{Nt-1} (Q^{-1}P)^\ell Q^{-1} b_{\ell \bmod{t}}.
\end{equation}
Substituting this into (\ref{eq:rep}) gives
\begin{equation} \label{eq:pqadic}
\begin{aligned}[b]
    v &= (Q^{-1}P)^{Nt+k+1} p + \sum_{\ell=0}^{Nt-1} (Q^{-1}P)^{(k+1)+\ell} Q^{-1} b_{\ell \bmod{t}} + \sum_{j=0}^k (Q^{-1}P)^jQ^{-1}a_j \\
    &=: \big((b_{t-1} \cdots b_1 b_0)^\infty a_k \cdots a_1 a_0\big)_{P/Q}.    
\end{aligned}
\end{equation}
We refer to the expansion in (\ref{eq:pqadic}) as the \textit{$P/Q$-adic representation} of $v$.
Thus, each $v \in \Z^d$ corresponds to a unique eventually left-periodic infinite word $\sigma = (b_{t-1} \cdots b_1 b_0)^\infty a_k \cdots a_1 a_0$ over $\D$ such that $(\sigma)_{P/Q} = v$. 
If $b_{t-1} = \cdots = b_1 = b_0 = 0$, then we identify $\sigma$ with the finite word $a_k \cdots a_1 a_0$. 

The remainder of this section is devoted to the case $d = 2$. Let $P,Q \in \Z^{2\times 2}$ be invertible matrices of the form
\begin{equation} \label{eq:pq_form}
P =
\begin{bmatrix}
    a & b \\
    \varepsilon & d
\end{bmatrix},
Q =
\begin{bmatrix}
    r & s \\
    0 & \delta
\end{bmatrix},
\end{equation}
where $a,b,d,r,s \in \Z$ with $r > 0$, and $\varepsilon,\delta \in \brc{1,-1}$. 
We further assume that $\det P$ and $\det Q$ are coprime (which implies by Proposition \ref{prop:coprime} that $P$ and $Q$ are themselves coprime). 
Denote by $\hat{Q} \in \Z^{2\times 2}$ the adjugate matrix of $Q$, so that $Q^{-1} = \hat{Q}/(\det Q)$. 
Let $\chi(x) = x^2 + \alpha x + \beta$ be the characteristic polynomial of $Q^{-1}P$.
Then $\alpha = -\tr(Q^{-1}P) = -[\tr(P \hat{Q})]/ (\det Q)$ and $\beta = \det(Q^{-1}P) = (\det P)/(\det Q)$. 
Hence, \[(\det Q)\chi(x) = (\det Q) x^2 - [\tr(P \hat{Q})]x + (\det P)\] is a primitive polynomial with integer coefficients. 

Now, by Lemma \ref{lem:subgp}, we have $(\Z^2 + P\Z^2[Q^{-1}P])/(P\Z^2[Q^{-1}P]) \cong \Z^2/P\Z^2$. 
Hence, any complete residue system of $\Z^2/P\Z^2$ also serves as a complete residue system for $(\Z^2 + P\Z^2[Q^{-1}P]) / P\Z^2[Q^{-1}P]$. 
The following proposition shows that $\D = \{(0,0)^\top, (1,0)^\top, \dots, (\abs{\det P}-1,0)^\top\}$ is such a complete residue system,
and this will be the digit set considered throughout the remainder of the discussion.

\begin{proposition} \label{prop:crs_mat}
Suppose $P = \begin{bmatrix} a & b \\ \varepsilon & d \end{bmatrix} \in \Z^{2\times 2}$ is invertible with $\varepsilon \in \brc{\pm 1}$, and $p = \abs{\det P}$. Then $\D = \{(0,0)^\top, (1,0)^\top, \dots, (p-1,0)^\top\}$ is a complete residue system for $\Z^2/P\Z^2$.
\end{proposition}
\begin{proof}
Let $v \in \Z^2$. Since the matrix $U =
\begin{bmatrix}
    1 & a \\
    0 & \varepsilon
\end{bmatrix}
$ is unimodular, its columns $\e_1$ and $P\e_1$ form a basis for~$\Z^2$. 
Thus $v = m\e_1 + P(n\e_1)$ for some $m,n \in \Z$. 
Let $\chi_P(x)$ denote the characteristic polynomial of~$P$. Since 
\[\chi_P(P) \e_1 = P^2 \e_1 - (\tr P)P\e_1 + (\det P)\e_1 = \boldsymbol{0},\]
we may repeatedly use this identity to reduce $m$ modulo $p$, thereby writing $v + P\Z^2 = m\e_1 + P\Z^2$, where now $m \in \brc{0,1,\dots,p-1}$.
Hence, $v + P\Z^2\in \D + P\Z^2$, which shows that $\D$ contains a complete residue system modulo $P\Z^2$. 
Since $\abs{\D} = \abs{\Z^2/P\Z^2} = p$, it follows that $\D$ is a complete residue system for $\Z^2/P\Z^2$. 
\end{proof}

We now describe the setting for our main result. Suppose $P,Q \in \Z^{2 \times 2}$ are of the form (\ref{eq:pq_form}) with coprime determinants. 
In addition, assume that $Q^{-1}P$ is expanding with characteristic polynomial $\chi(x) = x^2 + \alpha x + \beta$. 
This ensures that the attractor of the dynamical system corresponding to the digit system $(P,Q,\D)$ is finite. 
Moreover, if $p=\abs{\det P}$, then $p>\abs{\det{Q}}=r$ and $p$ is coprime to $r$. 
Under these assumptions, we may now state the main result of this section. 
The remainder of this section is devoted to its proof.

\begin{theorem} \label{thm:automaton}
Suppose $P,Q \in \Z^{2\times 2}$ are of the form (\ref{eq:pq_form}) with coprime determinants and $p = \abs{\det P}$. 
Let $\D = \{(0,0)^\top, (1,0)^\top, \dots, (p-1,0)^\top\}$, $\chi(x) = x^2 + \alpha x + \beta$ be the characteristic polynomial of the expanding matrix $Q^{-1}P$, and $\A$ be the attractor of the digit system $(P,Q,\D)$.
\begin{enumerate}
    \item If $0 \leq \alpha \leq \beta-1$, then $\A = \brc{(0,0)^\top}$.
    \item If $1 < -\alpha \leq \beta - 1$, then $\A = \brc{(0,0)^\top, \big(1- \frac{d}{\delta}, \frac{\varepsilon}{\delta}\big)^\top, \cdots, K\big(1- \frac{d}{\delta}, \frac{\varepsilon}{\delta}\big)^\top}$, where $K$ is the largest nonnegative integer such that $K(p + r\alpha + r) \leq p-1$.
    \item If $0 < -\alpha \leq 1$ and $-\alpha \leq \beta - 1$, then $\A = \brc{(0,0)^\top}$.
    \item If $0 < -\alpha \leq -\beta-1$, then $\A$ contains the set $\brc{(0,0)^\top, (-1,0)^\top, \big(\tfrac{d}{\delta}, -\tfrac{\varepsilon}{\delta}\big)^\top}$.
    \item If $0 < \alpha \leq -\beta-1$, then $\A$ contains $\brc{(0,0)^\top, -\big(1- \tfrac{d}{\delta}, \tfrac{\varepsilon}{\delta}\big)^\top, \dots, -K\big(1- \tfrac{d}{\delta}, \tfrac{\varepsilon}{\delta}\big)^\top}$, where $K$ is the largest nonnegative integer such that $K(p - r\alpha - r) \leq p-1$.
\end{enumerate}
\end{theorem}

Theorem \ref{thm:automaton} is an analogue of \cite[Theorem 2.1]{thuswaldner} in the present setting. 
As mentioned earlier, Theorem~\ref{thm:automaton} will be proved by constructing a finite letter-to-letter transducer 
that takes as input the $P/Q$-digit expansion of a vector $v \in \Z^2$ and outputs the expansion of $v \pm (1,0)^\top$. 
We recall below the definition of a finite letter-to-letter transducer. 

\begin{definition}
A \textit{finite letter-to-letter transducer} or \textit{automaton} is an ordered 6-tuple $\mathcal{T} = (\mbf{Q},\mathcal{A},\mathcal{B},\mbf{E},\mbf{I},\mbf{T})$, where:
\begin{itemize}
	\item $\mbf{Q}$ is a finite nonempty set whose elements are called \textit{states};
	\item $\mathcal{A}$ and $\mathcal{B}$ are finite sets of \textit{digits}, called the \textit{input} and \textit{output alphabets}, respectively;
	\item $\mbf{E} \subseteq \mbf{Q} \times \mathcal{A} \times \mathcal{B} \times \mbf{Q}$ is called the \textit{transition relation}; and
	\item $\mbf{I}$ and $\mbf{T}$ are nonempty subsets of $\mbf{Q}$, called the sets of \textit{initial} and \textit{terminal} states, respectively.
\end{itemize}
\end{definition}
A transducer $\mathcal{T}$ may be represented by a directed edge-labeled graph known as a \textit{transition diagram}. 
Here, $\mbf{Q}$ is the set of vertices, and a directed edge is drawn from a vertex $q_1$ to vertex $q_2$ with label $a \mid b$ if and only if $(q_1,a,b,q_2) \in \mbf{E}$. 
Such an edge is denoted by \[q_1 \xrightarrow{a \mid b} q_2.\] 
A sequence of the form
\[
\mathcal{P} : q_0 \xrightarrow{a_0 \mid b_0} q_1 \xrightarrow{a_1 \mid b_1} q_2 \xrightarrow{a_2 \mid b_2} \cdots \xrightarrow{a_{n-1} \mid b_{n-1}} q_n \xrightarrow{a_n \mid b_n} \cdots
\]
where $q_i \in \mbf{Q}$, $(a_i,b_i) \in \mathcal{A} \times \mathcal{B}$, and $(q_i, a_i, b_i, q_{i+1}) \in \mbf{E}$ for all $i\ge 0$, is called a \textit{walk} in the transducer $\mathcal{T}$. 
The infinite word $\mbf{a} = \cdots a_n a_{n-1} \cdots a_1 a_0 \in \mathcal{A}^\N$ is called the \textit{input} of $\mathcal{P}$, 
and the word $\mbf{b} = \cdots b_n b_{n-1} \cdots b_1 b_0 \in \mathcal{B}^\N$ is called its corresponding \textit{output}. 
The pair $\mbf{a} \mid \mbf{b}$ is called the \textit{label} of $\mathcal{P}$.
For further background on automata theory and transducers, we refer the reader to \cite{allouche,hopcroft,sakarovitch}.

 The equation
\[
(Q^{-1}P)^2 + \alpha(Q^{-1}P) + \beta I_2 = \boldsymbol{0}
\]
will be used extensively to handle the carries that occur when adding $(1,0)^\top$. 
Consequently, the conditions imposed on $\alpha$ and $\beta$ affect the resulting automata, 
and the transition diagrams differ across the five cases of Theorem \ref{thm:automaton}.

Let $\D^\N$ denote the set of all left-infinite words whose digits are in $\D$. 
If $\sigma \in \D^\N$ is eventually left-periodic, let $\sigma^{\pm A}$, $\sigma^{\pm B}$, $\sigma^{\pm C}$, and $\sigma^{\pm D}$ denote the eventually left-periodic words in $\D^\N$ defined by
\begin{align*}
    (\sigma^{\pm A})_{P/Q} &= (\sigma)_{P/Q} \pm (1,0)^\top, \\
    (\sigma^{\pm B})_{P/Q} &= (\sigma)_{P/Q} \pm (Q^{-1}P)(1,0)^\top \pm (\alpha - 1, 0)^\top, \\
    (\sigma^{\pm C})_{P/Q} &= (\sigma)_{P/Q} \mp (Q^{-1}P)(1,0)^\top \mp (\alpha,0)^\top, \text{and} \\
    (\sigma^{\pm D})_{P/Q} &= (\sigma)_{P/Q} \pm (Q^{-1}P)(1,0)^\top \pm (\alpha + 1,0)^\top.
\end{align*}
These operations determine the states of the transducers in all cases of Theorem \ref{thm:automaton}.
Moreoever, the set $\D = \brc{(0,0)^\top,(1,0)^\top,\dots,(p-1,0)^\top}$, where $p = \abs{\det P}$, serves as both the input and output alphabet. 
To identify the transition relations, we determine the effects of the above operations on the rightmost digit of a given digit string $\sigma$.

\subsection{The case $0 \leq \alpha \leq \beta-1$} \label{subsec:automaton1}
For our first case, we assume that the coefficients $\alpha = -\tr(Q^{-1}P)$ and $\beta = \det(Q^{-1}P)$ of the characteristic polynomial $\chi(x)$ are both positive and satisfy $\alpha \leq \beta-1$.
Under these assumptions, we show that the attractor of the digit system $(P,Q,\D)$ is $\brc{\boldsymbol{0}}$. 
Since $p=r\beta$, we have $p> r\alpha$. 
Let $\nu \in \brc{0,1,\dots,p-1}$. 
Following the notation of \cite{thuswaldner}, we identify $(\nu,0)^\top$ with $\nu$, and a horizontal bar over an expression indicates that the expression is interpreted as a single digit. 
A direct computation yields the following relations:
\begin{align*}
    (\sigma\nu)^{\pm A} &=
    \begin{cases}
        \sigma \ol{\nu \pm r}, & 0 \leq \nu \pm r < p \\
        \sigma^{\pm C} \, \ol{\nu \pm r \mp p}, & 0 \leq \nu \pm r \mp p < p,
    \end{cases} \tageq \label{eq:relnA1} \\
    (\sigma\nu)^B &=
    \begin{cases}
        \sigma^A \, \ol{\nu + r\alpha - r}, & r-r\alpha \leq \nu < p-r\alpha+r \\
        \sigma^{-B} \, \ol{\nu + r\alpha - r - p}, & p-r\alpha+r \leq \nu < p \\
        \sigma^{D} \, \ol{\nu + r\alpha - r + p}, & 0 \leq \nu < r-r\alpha, \\
    \end{cases} \tageq \label{eq:relnB1} \\
    (\sigma\nu)^{-B} &=
    \begin{cases}
        \sigma^{-A} \, \ol{\nu - r\alpha + r}, & r\alpha-r \leq \nu < p + r\alpha - r \\
        \sigma^B \, \ol{\nu - r\alpha + r + p}, & 0 \leq \nu < r\alpha-r \\
        \sigma^{-D} \, \ol{\nu - r\alpha + r - p}, & p + r\alpha - r \leq \nu < p,
    \end{cases} \tageq \label{eq:relnnegB1} \\
    (\sigma\nu)^{\pm C} &=
    \begin{cases}
        \sigma^{\mp A} \, \ol{\nu \mp r\alpha}, & 0 \leq \nu \mp r\alpha < p\\
        \sigma^{\pm B} \, \ol{\nu \mp r\alpha \pm p}, & 0 \leq \nu \mp r\alpha \pm p < p,         
    \end{cases} \\
    (\sigma\nu)^{\pm D} &=
    \begin{cases}
        \sigma^{\pm A} \, \ol{\nu \pm r\alpha \pm r}, & 0 \leq \nu \pm r\alpha \pm r < p \\
        \sigma^{\mp B} \, \ol{\nu \pm r\alpha \pm r \mp p}, & 0 \leq \nu \pm r\alpha \pm r \mp p < p.
    \end{cases}
\end{align*}

We prove only (\ref{eq:relnA1}) for $(\sigma\nu)^{A}$. The remaining relations are obtained similarly. 
By definition,
\begin{align*}
	\big((\sigma\nu)^A\big)_{P/Q} &= (Q^{-1}P)(\sigma)_{P/Q} + Q^{-1}(\nu,0)^\top + (1,0)^\top \\
	&= (Q^{-1}P)(\sigma)_{P/Q} + Q^{-1}(\nu,0)^\top + Q^{-1}(r,0)^\top \\
	&= (Q^{-1}P)(\sigma)_{P/Q} + Q^{-1}(\nu+r,0)^\top.
\end{align*}
If $\nu < p-r$, then $\nu + r \in \D$ and so $(\sigma\nu)^A = \sigma \, \ol{\nu+r}$. 
Otherwise, since $\nu+r \notin \D$, a carry occurs in the addition by $(1,0)^\top$ corresponding to the identity 
\[Q^{-1}(p,0)^\top = Q^{-1}(r\beta,0)^\top = (\beta,0)^\top = -(Q^{-1}P)^2(1,0)^\top - (Q^{-1}P)(\alpha,0)^\top.\] 
Thus,
\begin{align*}
	\big((\sigma\nu)^A\big)_{P/Q} &= (Q^{-1}P)(\sigma)_{P/Q} + Q^{-1}(\nu+r-p,0)^\top + Q^{-1}(p,0)^\top \\
	&= (Q^{-1}P)(\sigma)_{P/Q} + Q^{-1}(\nu+r-p,0)^\top -(Q^{-1}P)^2(1,0)^\top - (Q^{-1}P)(\alpha,0)^\top \\
	&= (Q^{-1}P)\big[ (\sigma)_{P/Q} - (Q^{-1}P)(1,0)^\top - (\alpha,0)^\top \big] + Q^{-1}(\nu+r-p,0)^\top \\
	&= (Q^{-1}P) (\sigma^C)_{P/Q} + Q^{-1}(\nu+r-p,0)^\top,
\end{align*}
with $0 \leq \nu + r - p < p$. Thus $(\sigma \nu)^A = \sigma^C \, \ol{\nu+r-p}$. 

Observe that (\ref{eq:relnB1}) and (\ref{eq:relnnegB1}), corresponding to $(\sigma\nu)^B$ and $(\sigma\nu)^{-B}$, each give three possible relations. 
However, depending on the sign of $r-r\alpha$, at most two of these relations can occur.
Thus, the resulting automaton depends on whether $\alpha>1$, $\alpha=1$, or $\alpha<1$.
An example of the case $\alpha > 1$ is shown in Figure~\ref{fig:automaton11}. 
It models the addition of $(1,0)^\top$ in the digit system $(P,Q,\D)$ where $P = \begin{bmatrix*}[r] -2 & 6 \\ 1 & 1\end{bmatrix*}$ and $Q = \begin{bmatrix*}[r] 3 & 1 \\ 0 & -1 \end{bmatrix*}$. 
In this case, $Q^{-1}P = \begin{bmatrix*}[r] -\tfrac{1}{3} & \tfrac{7}{3} \\ -1 & -1 \end{bmatrix*}$ and $\alpha = -\tr(Q^{-1}P) = 4/3 > 1$. 

To obtain the digit string $\sigma^a$ for $a \in \brc{\pm A, \pm B, \pm C, \pm D}$, we feed $\sigma$ into the transducer, starting at the state labeled $a$. 
The transducer then reads the digits of $\sigma$ from right to left, following edges whose input labels match the digit being read.
Each edge is labeled $j \mid k$, where $j$ is the input digit and $k$ is the corresponding output digit.

Let us consider, for instance, the automaton in Figure \ref{fig:automaton11}. 
Suppose $v \in \Z^2$ has digit representation $\sigma$. 
Then obtaining the digit representation of $v+(1,0)^\top$ is equivalent to determining $\sigma^A$. 
To do so, the digits of $\sigma$ are fed into the transducer and read from right to left, following the edge whose input label matches the digit being read. 
If, while reading the digits of $\sigma$, the automaton reaches a terminal state indicated by a black dot, 
then all remaining digits of $\sigma$ are copied unchanged to the output string. 
The same procedure is used to obtain the $P/Q$-expansion of $v-(1,0)^\top$, except that the representation of $v$ is fed into the automaton starting at the state labeled $\ol{A}$.

\begin{figure}[h!]
	\begin{center}
		\includegraphics{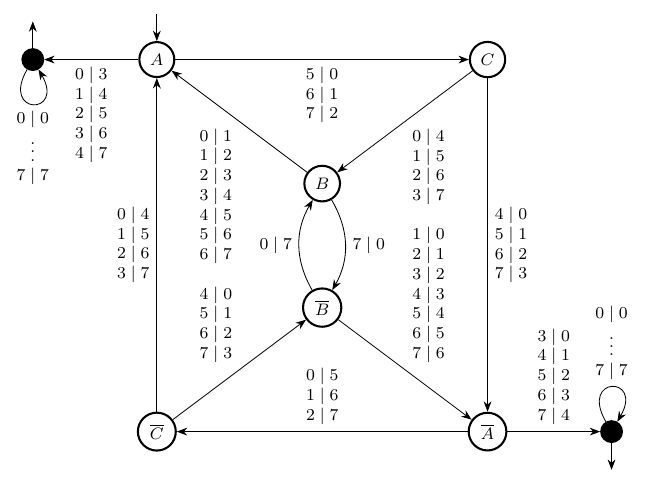}
		\caption{Transducer for addition by $(1,0)^\top$, where $P = \begin{bmatrix*}[r] -2 & 6 \\ 1 & 1\end{bmatrix*}$ and $Q = \begin{bmatrix*}[r] 3 & 1 \\ 0 & -1 \end{bmatrix*}$.} 
		\label{fig:automaton11}
	\end{center}
\end{figure}

Suppose $v \in \Z^2$ has finite digit representation $\sigma = 0^\infty d_{\ell-1} \cdots d_1 d_0$, where $d_{\ell-1} \neq 0$, 
and that $\sigma$ is fed into the transducer at the state labeled $A$ or $\ol{A}$. 
If a terminal state is reached after all nonzero digits of $\sigma$ have been read, 
then all remaining digits are copied unchanged to the output, and hence a finite digit string is produced. 

If, on the other hand, the automaton is in a nonterminal state after the nonzero digits have been read, then the remaining input digits are all zero. 
It therefore suffices to examine the walks in the automaton whose input digits are all zero.
As can be seen from Figure \ref{fig:automaton11}, every such walk reaches a terminal state after at most three steps. 
Consequently, all remaining zeros are copied unchanged to the output, and a finite digit string is again produced.

Applying this argument to $v = (0,0)^\top$, it follows that $(n,0)^\top$ has a finite $P/Q$-expansion for every $n \in \Z$. 
On the other hand, the operations $\pm C$ correspond to adding the vector $\pm (d/\delta, -\varepsilon/\delta)^\top \in \Z^2$, 
and therefore every vector of the form $m(d/\delta, -\varepsilon/\delta)^\top$, where $m\in\Z$, also has a finite $P/Q$-expansion. 
Consequently, every integer vector of the form $(n,0)^\top + m(d/\delta, -\varepsilon/\delta)^\top$, where $m,n \in \Z$, has a finite $P/Q$-expansion. 
Since $\brc{(1,0)^\top, (d/\delta, -\varepsilon/\delta)^\top}$ forms a basis of the lattice $\Z^2$, every vector in $\Z^2$ has a finite $P/Q$-expansion over $\D$. 
It follows that the attractor of the digit system $(P,Q,\D)$ is $\brc{(0,0)^\top}$.
The same arguments apply to the automaton corresponding to the case $\alpha=1$. 

The case $\alpha < 1$ gives rise to a different automaton, shown in Figure \ref{fig:automaton12}.
The figure corresponds to $P = \begin{bmatrix*}[r] 0 & 5 \\ 1 & 1\end{bmatrix*}$ and $Q = \begin{bmatrix*}[r] 2 & 1 \\ 0 & -1 \end{bmatrix*}$. 
In this case, we have $Q^{-1}P = \begin{bmatrix*}[r] \frac{1}{2} & 3 \\ -1 & -1 \end{bmatrix*}$ and $\alpha = -\tr(Q^{-1}P) = \tfrac{1}{2} < 1$. 
As with the transducer in Figure \ref{fig:automaton11}, every walk in this automaton whose input digits are all zero reaches a terminal state, after at most four steps. 
Hence the preceding argument applies here as well, and it follows that the attractor of the digit system $(P,Q,\D)$ is also $\brc{(0,0)^\top}$.
This completes the proof of the first case of Theorem~\ref{thm:automaton}.

\begin{figure}[ht!]
	\begin{center}
		\includegraphics{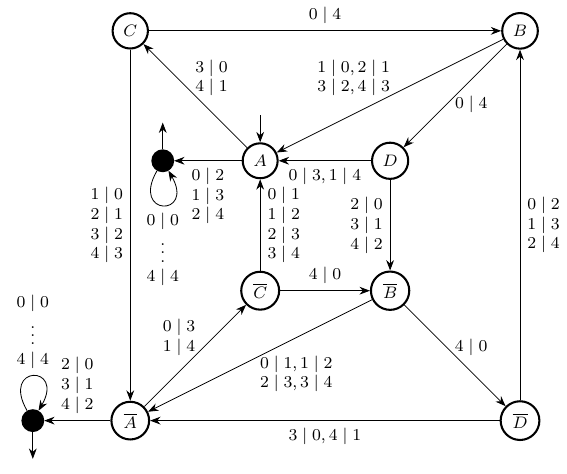}
	
		\caption{Transducer for addition by $(1,0)^\top$, where $P = \begin{bmatrix*}[r] 0 & 5 \\ 1 & 1\end{bmatrix*}$ and $Q = \begin{bmatrix*}[r] 2 & 1 \\ 0 & -1 \end{bmatrix*}$.} 
		\label{fig:automaton12}
	\end{center}
\end{figure}

We conclude this section by noting that the same arguments apply when $\alpha > 1$ and $\beta - 1 < \alpha < \beta$; the corresponding transducer is identical to that shown in Figure \ref{fig:automaton11}.
The case $\alpha \leq 1$ and $\beta - 1 < \alpha < \beta$, however, requires additional analysis, as new states may arise from $\pm D$.

\subsection{The case $0 < -\alpha \leq \beta-1$} \label{subsec:automaton2}
For this case, we have $p = r\beta \geq r - r\alpha >r$. The equations relating the operations $\pm A, \pm B, \pm C, \pm D$ are as follows:

{\allowdisplaybreaks
	\begin{align*}
	    (\sigma\nu)^{\pm A} &=
	    \begin{cases}
	        \sigma \ol{\nu \pm r}, & 0 \leq \nu \pm r < p \\
	        \sigma^{\pm C} \, \ol{\nu \pm r \mp p}, & 0 \leq \nu \pm r \mp p < p
	    \end{cases} \\
	    (\sigma\nu)^{\pm B} &=
	    \begin{cases}
	        \sigma^{\pm A} \, \ol{\nu \pm r\alpha \mp r}, & 0 \leq \nu \pm r\alpha \mp r < p \\
	        \sigma^{\pm D} \, \ol{\nu \pm r\alpha \mp r \pm p}, & 0 \leq \nu \pm r\alpha \mp r \pm p < p
	    \end{cases} \\
	    (\sigma\nu)^{\pm C} &=
	    \begin{cases}
	        \sigma^{\mp A} \, \ol{\nu \mp r\alpha}, & 0 \leq \nu \mp r\alpha < p \\
	        \sigma^{\mp D} \, \ol{\nu \mp r\alpha \mp p}, & 0 \leq \nu \mp r\alpha \mp p < p
	    \end{cases} \\
	    (\sigma\nu)^D &=
	    \begin{cases}
	        \sigma^A \, \ol{\nu + r\alpha + r}, & -r\alpha - r \leq \nu < p - r\alpha - r \\
	        \sigma^{-B} \, \ol{\nu + r\alpha+r-p}, & p - r\alpha - r \leq \nu < p \\
	        \sigma^D \, \ol{\nu + r\alpha+r+p} & 0 \leq \nu < -r\alpha - r
	    \end{cases} \\
	    (\sigma\nu)^{-D} &=
	    \begin{cases}
	        \sigma^{-A} \, \ol{\nu - r\alpha - r}, & r\alpha + r \leq \nu < p + r\alpha + r \\
	        \sigma^{B} \, \ol{\nu - r\alpha-r+p}, & 0 \leq \nu < r\alpha + r \\
	        \sigma^{-D} \, \ol{\nu - r\alpha-r-p} & p + r\alpha + r \leq \nu < p.
	    \end{cases}
	\end{align*}
}

Similar to the first case, at most two of the three possible relations for $D$ and $-D$ can occur,
depending on whether $\alpha < -1$, $\alpha=-1$, or $\alpha>-1$. 
An example of the case where $\alpha < -1$ is shown in Figure \ref{fig:automaton21}. 
It realizes addition by $(1,0)^\top$ for $P = \begin{bmatrix*}[r] 3 & -4 \\ 1 & 1\end{bmatrix*}$ and $Q = \begin{bmatrix*}[r] 2 & 1 \\ 0 & 1 \end{bmatrix*}$. 
Observe that the states $B$ and $-B$ play no role in the automaton, since no paths lead to them. 

\begin{figure}[H]
	\begin{center}
		\includegraphics{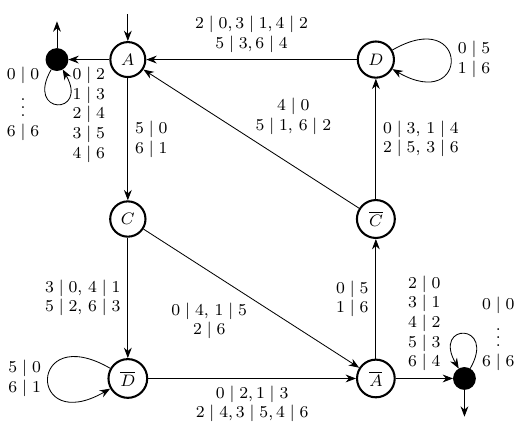}
	
		\caption{Transducer for addition by $(1,0)^\top$, where $P = \begin{bmatrix*}[r] 3 & -4 \\ 1 & 1\end{bmatrix*}$ and $Q = \begin{bmatrix*}[r] 2 & 1 \\ 0 & 1 \end{bmatrix*}$.}
		\label{fig:automaton21}
	\end{center}
\end{figure}

Let us first consider the case $\alpha < -1$. 
Of particular interest in the transition diagram of Figure~\ref{fig:automaton21} is the self-loop at state $D$ having input label $0$. 
This implies the existence of a finite string $\sigma$ for which $\sigma^A$ is infinite and left-periodic. 
The vectors represented by such infinite words are among the nonzero elements of the attractor $\A$ of the digit system $(P,Q,\D)$.

To determine the elements of $\A$ explicitly, let $\gamma := r\alpha + r + p$, which is the output digit corresponding to the input digit $0$ on the self-loop at state $D$.
Suppose $\sigma = 0^\infty \, d_{n-1} \cdots d_1 d_0$ is a finite input string. 
If $\sigma$ is fed into the automaton from any initial state, then one of the following occurs:
\begin{itemize}
	\item the resulting walk ends at a terminal state, in which case the output is finite, of the form $0^\infty u$ for some finite word $u$ over $\mathcal{D}$.
	\item the resulting walk ends at the self-loop at $D$, in which case the output is infinite and left-periodic, of the form $\gamma^\infty v$ for some finite word $v$ over $\mathcal{D}$.
\end{itemize}
Consequently, $(0^\infty)_{P/Q} = (0,0)^\top$ and $(\gamma^\infty)_{P/Q}$ are both elements of $\A$. 

Applying the same reasoning to an input string of the form $\gamma^\infty \, d_{n-1} \cdots d_1 d_0$, 
we obtain three possible output forms: $0^\infty u$, $\gamma^\infty v$, or $(\ol{2\gamma})^\infty w$ for some finite words $v,u,w$ over $\mathcal{D}$.

Let $K$ be the largest nonnegative integer satisfying $K\gamma \leq p-1$.
More generally, given $\ell \in \brc{0,1,\dots,K}$, an input string of the form $(\ol{\ell\gamma})^\infty d_{n-1} \cdots d_1d_0$ produces an output string of one of the following forms:
\begin{itemize}
    \item $(\ol{\ell\gamma})^\infty u$, for some finite word $u$ over $\mathcal{D}$, if the resulting walk ends at a terminal state;
    \item $(\ol{(\ell-1)\gamma})^\infty v$, for $\ell > 0$ and some finite word $v$ over $\mathcal{D}$, if the resulting walk ends at the self-loop at~$\overline{D}$;
    \item $(\ol{(\ell+1)\gamma})^\infty w$, for $\ell < K$ and some finite word $w$ over $\mathcal{D}$, if the resulting walk ends at the self-loop at~$D$.
\end{itemize}

Applying the automaton successively to $(0,0)^\top$ from the initial states $\pm A$ and $\pm C$ exhausts all vectors of $\Z^2$, as discussed in the first case. 
Hence, every element of the attractor $\A$ is of the form $\big((\ol{\ell\gamma})^\infty\big)_{P/Q}$, where $\ell \in \brc{0,1,\dots,K}$. 
To determine the vector represented by $\big((\ol{\ell\gamma})^\infty\big)_{P/Q}$, 
we apply the transducer successively $\ell$ times to $(0,0)^\top$, each time at state $D$. By definition, $\sigma^D$ performs addition by
\[
(Q^{-1}P)(1,0)^\top + (\alpha+1,0)^\top = \Big(1- \tfrac{d}{\delta}, \tfrac{\varepsilon}{\delta}\Big)^\top.
\]
Therefore, $\A = \big\{(0,0)^\top, \big(1- \tfrac{d}{\delta}, \tfrac{\varepsilon}{\delta}\big)^\top, \ldots, K\big(1- \tfrac{d}{\delta}, \tfrac{\varepsilon}{\delta}\big)^\top\big\}$. 
This completes the proof of the second statement of Theorem \ref{thm:automaton}.

The case $\alpha \geq -1$ gives rise to a different transducer, shown in Figure~\ref{fig:automaton22}. 
The figure corresponds to $P = \begin{bmatrix*}[r] 4 & -1 \\ 1 & 1\end{bmatrix*}$ and $Q = \begin{bmatrix*}[r] 2 & 5 \\ 0 & 1 \end{bmatrix*}$. 
The situation is similar to that of Figure \ref{fig:automaton12}: every walk in the transducer whose remaining input digits are all zero reaches a terminal state after at most five readings. 
Consequently, every finite input word produces a finite output word, and it follows that the attractor of the digit system $(P,Q,\mathcal{D})$ is $\A = \brc{(0,0)^\top}$. 
The same arguments apply to the automaton corresponding to the case $\alpha=-1$.
This proves the third statement of Theorem~\ref{thm:automaton}.
\begin{figure}[H]
	\begin{center}
		\includegraphics{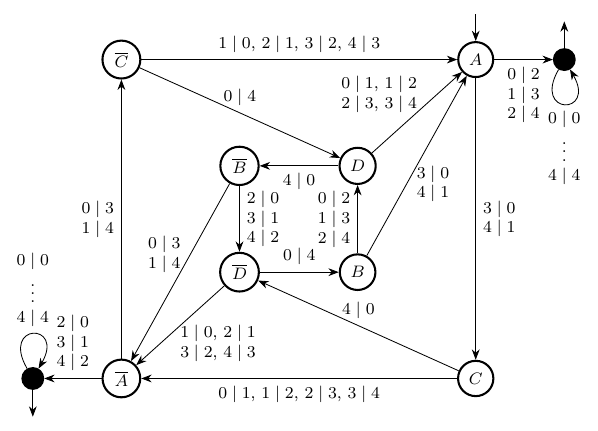}

		\caption{Transducer for addition by $(1,0)^\top$, where $P = \begin{bmatrix*}[r] 4 & -1 \\ 1 & 1\end{bmatrix*}$ and $Q = \begin{bmatrix*}[r] 2 & 5 \\ 0 & 1 \end{bmatrix*}$.}
		\label{fig:automaton22}
	\end{center}
\end{figure}

\subsection{The case $0 < -\alpha \leq -\beta-1$} \label{subsec:automaton3}
In this case, $p = -r\beta \geq r - r\alpha$. The relations between the states are as follows:

\begin{align*}
	(\sigma\nu)^{\pm A} &=
	\begin{cases}
		\sigma \ol{\nu \pm r}, & 0 \leq \nu \pm r < p \\
		\sigma^{\mp C} \, \ol{\nu \pm r \mp p}, & 0 \leq \nu \pm r \mp p < p,
	\end{cases} \\
	(\sigma\nu)^{\pm B} &=
	\begin{cases}
		\sigma^{\pm A} \, \ol{\nu \pm r\alpha \mp r}, & 0 \leq \nu \pm r\alpha \mp r < p \\
		\sigma^{\mp B} \, \ol{\nu \pm r\alpha \mp r \pm p}, & 0\leq \nu\pm r\alpha \mp r\pm p <p, 
	\end{cases} \\
	(\sigma\nu)^{\pm C} &=
	\begin{cases}
		\sigma^{\mp A} \, \ol{\nu \mp r\alpha}, & 0 \leq \nu \mp r\alpha < p\\
		\sigma^{\pm B} \, \ol{\nu \mp r\alpha \mp p}, & 0 \leq \nu \mp r\alpha \mp p  < p.
	\end{cases}
 \end{align*}

Observe that the operations $D$ and $-D$ play no role in the above relations, just as in Subsection \ref{subsec:automaton1} for the case $\alpha \geq 1$. 
In contrast to the previous subsections, the automaton for this case is independent of the value of $\alpha$. 
Figure \ref{fig:automaton3} shows the transducer realizing the addition by $(1,0)^\top$ for the digit system $(P,Q,\D)$ where $P = \begin{bmatrix*}[r] 2 & -1 \\ 1 & -3 \end{bmatrix*}$ and $Q = \begin{bmatrix*}[r] 3 & -8 \\ 0 & 1 \end{bmatrix*}$.

\begin{figure}[H]
	\begin{center}
		\includegraphics{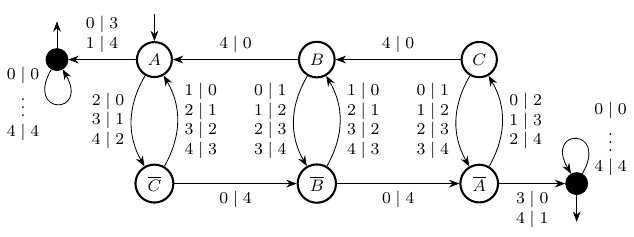}
		
		\caption{Transducer for addition by $(1,0)^\top$, where $P = \begin{bmatrix*}[r] 2 & -1 \\ 1 & -3 \end{bmatrix*}$ and $Q = \begin{bmatrix*}[r] 3 & -8 \\ 0 & 1 \end{bmatrix*}$.}
		\label{fig:automaton3}
	\end{center}
\end{figure} 

Suppose $v \in \Z^2$ has finite $P/Q$-expansion $\sigma = 0^\infty \, d_{\ell-1} \cdots d_1 d_0$. 
If $\sigma$ is fed into the automaton, the resulting walk either ends at a terminal state or eventually enters the cycle $\overline{A} \to C \to \overline{A}$. 
In the former case, the output is a finite string.
In the latter case, the output has one of the following forms: $\big(\ol{(p-r)} \, \ol{(-r\alpha)}\big)^\infty u$ 
or $\big(\ol{(-r\alpha)} \, \ol{(p-r)}\big)^\infty u$, where $u$ is a finite word over $\mathcal{D}$. 
Thus, $(0,0)^\top$, $\big(\big(\ol{(p-r)} \, \ol{(-r\alpha)}\big)^\infty\big)_{P/Q}$, and $\big(\big(\ol{(-r\alpha)} \, \ol{(p-r)}\big)^\infty\big)_{P/Q}$ are elements of the attractor $\A$.
The second and third vectors in this list are obtained by feeding the $P/Q$-expansion of $0$ into the automaton at the starting states $C$ and $A$, respectively. 
The operation $C$ corresponds to subtraction by $\big(-\tfrac{d}{\delta}, \tfrac{\varepsilon}{\delta}\big)^\top$, while the operation $-A$ corresponds to subtraction by $(1,0)^\top$. 
Thus, $\A$ contains
\[
\brc{(0,0)^\top, (-1,0)^\top, \big(\tfrac{d}{\delta}, -\tfrac{\varepsilon}{\delta}\big)^\top}.
\]
However, these are not the only elements of $\A$.
To see this, consider an input string of the form $\sigma=\big(\ol{(p-r)} \, \ol{(-r\alpha)}\big)^\infty u$ for some word $u$ over $\mathcal{D}$.
Depending on the values of $\alpha$ and $\beta$, the corresponding walk may end at a terminal state or eventually enter one of the two-cycles. 
If the walk ends at a terminal state, then the output has the same form as $\sigma$.
If the walk eventually enters one of the two-cycles, then the output is one of the following: a finite string, a string having the same form as $\sigma$, $(\overline{-r\alpha-r+p})^{\infty}$, 
$\big(\overline{(2p-2r+r\alpha)}\,\overline{(-2r\alpha+r-p)}\big)^{\infty}$, and $\big(\overline{(2p-2r)}\,\overline{(-2r\alpha)}\big)^{\infty}$. 

For example, consider Figure \ref{fig:automaton3}. 
The vectors $(0,0)^\top$, $(\ol{(p-r)} \, \ol{(-r\alpha)})_{P/Q}=(21)^{\infty}_{P/Q}$ and $(12)^{\infty}_{P/Q}$ are elements of $\mathcal{A}$.
Feeding the string $(21)^\infty0002$ into the automaton yields the output $(42)^\infty 2440$. 
Hence, $\big((42)^\infty\big)_{P/Q}$ is also an element of $\mathcal{A}$.

\subsection{The case $0 <\alpha \leq -\beta-1$} \label{subsec:automaton4}
In this case, $p = -r\beta \geq r + r\alpha$, and the relations between the states are as follows:

\begin{align*}
	(\sigma\nu)^{\pm A} &=
    \begin{cases}
        \sigma \ol{\nu \pm r}, & 0 \leq \nu \pm r < p \\
        \sigma^{\mp C} \, \ol{\nu \pm r \mp p}, & 0 \leq \nu \pm r \mp p < p,
    \end{cases} \\
    (\sigma\nu)^{\pm C} &=
    \begin{cases}
        \sigma^{\mp A} \, \ol{\nu \mp r\alpha}, & 0 \leq \nu \mp r\alpha < p \\
        \sigma^{\mp D} \, \ol{\nu \mp r\alpha \pm p}, & 0 \leq \nu \mp r\alpha \pm p < p,
    \end{cases} \\
    (\sigma\nu)^{\pm D} &=
    \begin{cases}
        \sigma^{\pm A} \, \ol{\nu \pm r\alpha \pm r}, & 0 \leq \nu \pm r\alpha \pm r < p \\
        \sigma^{\pm D} \, \ol{\nu \pm r\alpha \pm r \mp p}, & 0 \leq \nu \pm r\alpha \pm r \mp p < p.
    \end{cases}
\end{align*}

From these relations, we see that there are no incoming edges to the states $B$ and $\overline{B}$, regardless of the value of $\alpha$. 
Consequently, these states may be omitted from the corresponding transition diagram. 
An example of the automaton for this case is shown in Figure \ref{fig:automaton4}, 
corresponding to $P = \begin{bmatrix*}[r] -2 & -3 \\ -1 & 2 \end{bmatrix*}$ and $Q = \begin{bmatrix*}[r] 4 & -7 \\ 0 & 1 \end{bmatrix*}$.

\begin{figure}[H]
	\begin{center}
		\includegraphics{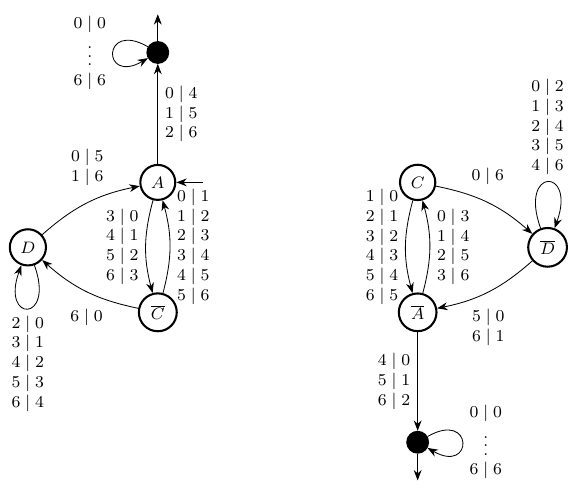}		
		
		\caption{Transducer for addition by $(1,0)^\top$, where $P = \begin{bmatrix*}[r] -2 & -3 \\ -1 & 2 \end{bmatrix*}$ and $Q = \begin{bmatrix*}[r] 4 & -7 \\ 0 & 1 \end{bmatrix*}$.}
		\label{fig:automaton4}
	\end{center}
\end{figure}

As in Figure \ref{fig:automaton21}, the transducer in Figure \ref{fig:automaton4} has a self-loop at the state $\overline{D}$ with input label $0$. 
In this case, the output digit on the self-loop is $\gamma := p-r\alpha-r$, and the proof proceeds in the same way as in Subsection~\ref{subsec:automaton2}. 
Let $K$ be the largest nonnegative integer satisfying $K\gamma \leq p-1$.
Then, for each $\ell \in \brc{0,1,\dots,K}$, the vector $\big(\ol{(\ell\gamma)}^\infty\big)_{P/Q}$ belongs to the attractor $\A$. 
Each of these vectors is obtained by feeding the $P/Q$-expansion of $(0,0)^\top$ into the automaton starting from the state $\overline{D}$ and repeating the process $\ell$ times. 
Because the operation $-D$ corresponds to subtracting the vector $\big(1- \frac{d}{\delta}, \frac{\varepsilon}{\delta}\big)^\top$, it follows that $\A$ contains
\[
\brc{(0,0)^\top, -\big(1- \tfrac{d}{\delta}, \tfrac{\varepsilon}{\delta}\big)^\top, \dots, -K\big(1- \tfrac{d}{\delta}, \tfrac{\varepsilon}{\delta}\big)^\top}.
\]
This finishes the proof of Theorem \ref{thm:automaton}.

As in Case 4 of Theorem \ref{thm:automaton}, $\A$ may contain vectors other than those obtained through the aforementioned process. 
This phenomenon occurs in the example shown in Figure \ref{fig:automaton4}, where $p-r-r\alpha = 2$. 
If the infinite word $2^\infty$ is fed to the transducer starting from the state $C$, the infinite periodic word $(51)^\infty$ is obtained. 
Hence, $\big((51)^\infty\big)_{P/Q}$ is another periodic point of $\A$, distinct from the vectors obtained above.

\end{document}